\documentclass[reqno,A4paper]{amsart}

\usepackage{amsmath}
\usepackage{amssymb}
\usepackage{amsthm}
\usepackage{enumerate}
\usepackage{mathrsfs} 
\usepackage{eqlist}
\usepackage{array}

\theoremstyle{plain}
\newtheorem{theorem}{Theorem}[section]
\newtheorem{prop}[theorem]{Proposition}
\newtheorem{lemma}{Lemma}[section]
\newtheorem{corol}{Corollary}[theorem]

\theoremstyle{definition}
\newtheorem{definition}{Definition}
\newtheorem{remark}{\textup{Remark}} 

\numberwithin{equation}{section}

\begin{document}

\title[ON THE RATIONAL PARAMETRIC SOLUTION OF DIAGONAL QUARTIC VARIETIES]%
{ON THE RATIONAL PARAMETRIC SOLUTION OF DIAGONAL QUARTIC VARIETIES}
\author[HASSAN SHABANI-SOLT \and AMIR SARLAK]%
{HASSAN SHABANI-SOLT* \and AMIR SARLAK**}

\newcommand{\acr}{\newline\indent}

\address{\llap{*\,}Department of Mathematics\acr
                   Urmia University\acr
                   Urmia\acr
                   IRAN}
\email{h.shabani.solt@gmail.com}

\address{\llap{**\,}School of Mathematics, Statistics and Computer Science, College of Science\acr
                    University of Tehran\acr
                    Tehran\acr
                    IRAN}
\email{sarlak$@$ut.ac.ir}




\begin{abstract}
In this paper we will exhibit a rational parametric solution for the Diophantine equations of diagonal quartic varieties. Our approach is based on utilizing the Calabi-Yau varieties including elliptic curves and diagonal quartic surfaces.
\end{abstract}

\maketitle

\section{Introduction}

Diagonal quartic varieties, are defined as a Diophantine equations in the form: 
\begin{equation}\label{equ1}
	X_4 \quad : \qquad \sum_{i=1}^{n} \alpha_i {x_i}^4=0 \qquad (\alpha_i \in \mathbb{Z}, \; \alpha_i \neq 0, \; n \ge 3)
\end{equation}
$X_4$ is a smooth algebraic variety in the projective space $\mathbb{P}^{n-1}$. In this paper by rational parametric solution of the diagonal quartic varieties of $X_4$, we mean a solution of this Diophantine equation as a rational function with coefficients of rational number field. Considering smooth hypersurface of $X \subseteq \mathbb{P}^{n}$, defined on the algebraically closed field of $k$ with characteristic zero and of the degree $d \ge 3$. It is shown that for the sufficiently great values of $n$, $X$ has an onto parametric solution on the field $k$ \cite{5}. In other words it is unirational, namely, for a positive integer $m$ it is dominated by $\mathbb{P}^{m}$. For the case $d=3$ the problem of unirationality has been solved, even for the rational number field \cite{7}. For the case $d=4$ and for all $n \ge 5$, even for the diagonal quartic varieties the problem of unirationality on the rational number field, is still open. But for $n=3$, $X_4$ is a smooth algebraic curve and it is of genus 3. Hence it can only have finitely many rational solutions and it does not have any rational parametric solution. And for $n=4$, we have the following diagonal quartic surface 
\begin{equation}\label{equ2}
	ax^4+by^4+cz^4+dw^4=0 
\end{equation}

Equation \ref{equ2} is a $K_3$ surface which is not unirational. However for some particular cases through considering specific conditions to the coefficients, the rational parametric solutions for equation \ref{equ2} exist. For instance the rational parametric solution using Richmond method \cite{8}, when 
the product of the coefficients $abcd$, is a perfect square. And the following equation
\begin{equation}
	x^4+ky^4=z^4+kw^4   
\end{equation}

for $k=1$ due to the classical result of Euler, the rational parametric solution exist. Moreover, for $k=4$ the rational parametric solution exist \cite{1}.
When $X_4$ is a symmetric diagonal varieties, that is: 
\begin{equation}
	X'_4 \quad : \qquad \sum_{i=1}^{m} \alpha_i {x_i}^4=\sum_{i=1}^{m} \alpha_i {y_i}^4
\end{equation}
It was demonstrated in \cite{2} that for $n= 2m \ge 6$, it is always possible to exhibit a rational parametric solution for $X'_4$.

In this paper using diagonal quartic surfaces and elliptic curves, we  study the rational parametric solution for some of asymmetric diagonal quartic varieties of $X_4$, and for the cases $n \ge 6$. We show that:\\ 

\begin{theorem}\label{th1}
	Considering the following diagonal quartic varieties of $X_4$ for $n=6$,
	\begin{equation}\label{equ16}
		x^4+y^4-z^4=2u^4+2v^4+2w^4
	\end{equation}
	\begin{equation}\label{equ17}
		x^4+y^4-z^4=u^4+2v^4+2w^4
	\end{equation}
	\begin{equation}\label{equ18}
		x^4+y^4+4z^4=u^4+2v^4+2w^4
	\end{equation}
	\begin{equation}\label{equ19}
		x^4+y^4+4z^4=4u^4+v^4-2w^4
	\end{equation}
	\begin{equation}\label{equ20}
		x^4+y^4-2z^4=2u^4+2v^4+2w^4
	\end{equation}
	the rational parametric solution for equation \ref{equ16} can be constructed in the form:
	\begin{equation*}
		\begin{split}
			x & = (\frac{\bar{x}}{\bar{y}})^4+2(a^4+4b^4) (\frac{\bar{x}}{\bar{y}})^2-(a^4+4b^4)^2\\
			y & = (\frac{\bar{x}}{\bar{y}})^4-2(a^4+4b^4) (\frac{\bar{x}}{\bar{y}})^2-(a^4+4b^4)^2\\
			z & =4ab(\frac{\bar{z} \bar{w} \bar{x}}{\bar{y}^3})\\
		\end{split}
		\qquad  \qquad
		\begin{split}
			u & = (\frac{\bar{x}}{\bar{y}})^4+(a^4+4b^4)^2\\
			v & =2a^2(\frac{\bar{z} \bar{w} \bar{x}}{\bar{y}^3})\\
			w & =4b^2(\frac{\bar{z} \bar{w} \bar{x}}{\bar{y}^3})\\
		\end{split}
	\end{equation*}
	and for equation \ref{equ17} the rational parametric solution can be constructed in the form:	
	\begin{equation*}
		\begin{split}
			x & = \bar{x}^2\\
			y & = (a^4+4b^4)\bar{y}^2\\
			z & = \bar{w}\bar{z}\\
		\end{split}
		\qquad  \qquad  
		\begin{split}
			u & = 2ab(\bar{x}\bar{y})\\
			v & = 2b^2(\bar{x}\bar{y})\\
			w & = a^2(\bar{x}\bar{y})\\
		\end{split}
	\end{equation*}
	In which we have: 
	\begin{equation*}
		\begin{split}
			\bar{x} & =a^8+24a^4b^4+16b^8\\
			\bar{y} & =4ab(4b^4-a^4)\\
			\bar{z} & =a^8+32a^2b^6-8a^4b^4+8a^6b^2+16b^8\\
			\bar{w} & =a^8-32a^2b^6-8a^4b^4-8a^6b^2+16b^8\\
		\end{split}
	\end{equation*}
	Applying the same approach for equations \ref{equ16} and \ref{equ17}, the rational parametric solution for the rest of the equations can be expressed analogously.
	
\end{theorem}

\begin{theorem}\label{th2}
	Considering the quartic Diophantine equation
	\begin{equation}\label{equ13}
		f(x,y)= \sum_{i=1}^{m} k_i {x_i}^4 
	\end{equation} 
	which $k_i s$ are rational numbers and $f(x,y)$ is one of the following expressions

	\begin{equation}\label{equ36}
		x^4-y^4, \qquad x^4+4y^4, \qquad 2x^4+2y^4,
	\end{equation}
	If equation \ref{equ13} has a rational parametric solution as bellow  
	\begin{equation}
		\begin{split}
			x &=a_1(t)\\
			y &=a_2(t)\\
			x_i &=b_i(t) \qquad (1 \le i \le m)
		\end{split}
	\end{equation}
	in which for any value of $t$, $a_1(t)$ and $a_2(t)$ are not zero, then there exist diagonal quartic varieties in the form of $X_4$ with $n=\frac{1}{2}(m^2+m+6)$ such that a rational parametric solution for them can be exhibited.

\end{theorem}

\begin{theorem}\label{th3}
	For any integer $m \ge 3$,  there exist diagonal varieties in the form of $X_4$ with $n=2m^2-m+4$, which the rational parametric solution can be exhibited for them.
	
\end{theorem}

\begin{theorem}\label{th4}
	Assume that there is a rational parametric solution for the Diophantine chain
	\begin{equation}\label{equ31}
		F_1(x_1,y_1)= F_2(x_2,y_2)= \dots = F_m(x_m,y_m)
	\end{equation}
	in which $F_i(x,y), \ (1 \le i \le m)$ is equal to one of the expressions:
	\[
	x^4+4y^4 \quad, \quad x^4-y^4 \quad, \quad 2x^4+2y^4
	\]
	Thus there is a rational parametric solution for the Diophantine chain
	\begin{equation}\label{equ32}
		\phi_1(x_1,y_1,z_1) =\phi_2(x_2,y_2,z_2)= \dots =\phi_m(x_m,y_m,z_m)
	\end{equation} 
	in which $\phi_i(x,y,z), \ (1 \le i \le m)$ is equal to one of the expressions:
	\begin{equation*}
		x^4+y^4-2z^4  \quad , \quad  x^4+y^4-z^4
	\end{equation*}
	
\end{theorem}

\section{Preliminaries}

In this part we will review some properties of elliptic curves and diagonal quartic surfaces.

\begin{definition}\label{def1}
	A positive integer $n$ is a congruent number if there exists a rational right angle triangle with area $n$.
\end{definition}

\begin{prop}\label{pro1}
	(\cite{6})The following statements are equivalent definitions for a positive integer n to be a congruent number.
	
	(i) There exist $x, y, z, t \in \mathbb{Z^+}$ satisfying the rationalized Diophantine equations 
	\begin{equation}\label{equ39}
		\begin{split}
			x^2+ny^2=z^2\\
			x^2-ny^2=t^2\\
		\end{split}
	\end{equation}
	
	(ii) There exist $x, y, z \in \mathbb{Z^+}$ satisfying the rationalized Diophantine equations 
	\begin{equation}\label{equ40}
		x^4-n^2y^4=z^2 
	\end{equation}
	
	(iii) The elliptic curve 
	\begin{equation}\label{equ41}
		Y^2=X^3-n^2X 
	\end{equation}
	has nontrivial solutions in rational numbers.
\end{prop}

\begin{remark}\label{rem1}
	In Proposition \ref{pro1}, it can be noticed that if $(x,y,z,t)$ satisfies the rationalized Diophantine equations for a congruent number $n$, that is:
	\begin{equation*}
		\begin{split}
			x^2+ny^2 & =z^2\\
			x^2-ny^2 & =t^2\\
		\end{split}
	\end{equation*}
	through multiplying above equations we have
	\[
	x^4-n^2y^4=(zt)^2
	\]
	putting 
	\[
	X=x, \quad Y=y, \quad K=zt, 
	\]
	we will have 
	\begin{equation}
		X^4-n^2Y^4=K^2
	\end{equation}
	also by putting 
	\[
	U=(\frac{X}{Y})^2=(\frac{x}{y})^2 \qquad , \qquad
	V=\frac{KX}{Y^3}=\frac{xzt}{y^3}
	\]
	the following expression is obtained :
	\begin{equation}
		V^2=U^3-n^2U
	\end{equation}
	
\end{remark}

\begin{prop}\label{pro2}
	
	we have:
	
	$(i)$ The number $n=a^4+4b^4 \quad (a,b \in \mathbb{Z^{+}})$ is a congruent number.
	
	$(ii)$ The number $n=2a^4+2b^4 \quad (a,b \in \mathbb{Z^{+}})$ is a congruent number.
	
	$(iii)$ The number $n=a^4-b^4 \quad (a,b \in \mathbb{Z^{+}})$ is a congruent number.
	
\end{prop}

\begin{proof}(i). Using Proposition \ref{pro1} part(i), It suffices to show that there exist $x,y \in \mathbb{Z^{+}}$ such that $x^2+ny^2$ and $x^2-ny^2$ are both squares. For $n=a^4+4b^4$ we have 
	\begin{equation}\label{equ4}
		x^2+ny^2= x^2+(a^4+4b^4)y^2\\
	\end{equation}
	\begin{equation}\label{equ5}
		x^2-ny^2= x^2-(a^4+4b^4)y^2\\
	\end{equation}
	let $x=a^8+24a^4b^4+16b^8$ and $y=4ab(4b^4-a^4)$ in equation \eqref{equ4}, as a result
	\begin{equation}\label{equ6}
		x^2+(a^4+4b^4)y^2= (a^8+24a^4b^4+16b^8)^2+16(a^4+4b^4)a^2b^2(4b^4-a^4)^2
	\end{equation}
	by factoring in equation \eqref{equ6}, it will result
	\begin{equation}
		x^2+(a^4+4b^4)y^2= (a^8+32a^2b^6-8a^4b^4+8a^6b^2+16b^8)^2
	\end{equation}
	similarly if the values $x=a^8+24a^4b^4+16b^8$ and $y=4ab(4b^4-a^4)$ are used in equation \eqref{equ5}, it will result 
	\begin{equation}
		x^2-(a^4+4b^4)y^2= (a^8-32a^2b^6-8a^4b^4-8a^6b^2+16b^8)^2
	\end{equation}
	let 
	\begin{equation}
		z= (a^8+32a^2b^6-8a^4b^4+8a^6b^2+16b^8)
	\end{equation}
	\begin{equation}
		t= (a^8-32a^2b^6-8a^4b^4-8a^6b^2+16b^8)
	\end{equation}
	Therefore, there exist $x, y, z, t \in \mathbb{Z^+}$ satisfying the rationalized Diophantine equations \ref{equ39}. Then according to Proposition \ref{pro1} part(i), $n=a^4+4b^4$ is a congruent number.
	
	(ii). The proof is exactly like the part (i). Using the same strategy, we will have the values 
	\begin{equation*}
		\begin{split}
			x & = a^8+6a^4b^4+b^8\\
			y & = 2ab(4b^4-a^4)\\
			z & = (a^8+4a^2b^6-2a^4b^4+4a^6b^2+b^8)\\
			t & = (a^8-4a^2b^6-2a^4b^4-4a^6b^2+b^8)\\
		\end{split}
	\end{equation*}
	Therefore, there exist $x, y, z, t \in \mathbb{Z^+}$ satisfying the rationalized Diophantine equations \ref{equ39}. Then according to Proposition \ref{pro1} part(i), $n=2a^4+2b^4$ is a congruent number.
	
	(iii). The proof is exactly like the part (i). Using the same strategy, we will have the values 
	\begin{equation*}
		\begin{split}
			x & = b^2(a^4+b^4)\\
			y & = 2ab^3\\
			z & = b^4(a^4+2a^2b^2-b^4)^2\\
			t & = b^4(a^4-2a^2b^2-b^4)\\
		\end{split}
	\end{equation*}
	Therefore, there exist $x, y, z, t \in \mathbb{Z^+}$ satisfying the rationalized Diophantine equations \ref{equ39}. Then according to Proposition \ref{pro1} part(i), $n=a^4-b^4$ is a congruent number.
\end{proof}

\begin{prop}\label{pro5}
	(\cite{4}, Proposition 6.5.6.) Let $n$ be a nonzero integer. The equation $x^4-y^4=nt^2$ has a solution with $xyt \ne 0$ if and only if $|n|$ is a congruent number. More precisely, if $x^4-y^4=nt^2$ with $xyt \ne 0$, then $v^2=u^3-n^2u$ with $(u,v)=(-ny^2/x^2,n^2yt/x^3)$, and conversely if $v^2=u^3-n^2u$ with $v \ne 0$, then $x^4-y^4=nt^2$, with $x=u^2+2nu-n^2$, $y=u^2-2nu-n^2$, and $t=4v(u^2+n^2)$.\\
\end{prop}

\begin{prop}\label{pro3}
	Let $n$ be a congruent number and let $(u,v), \;  v \ne 0$ be a rational point on a elliptic curve $v^2=u^3-n^2u$. Then
	\begin{equation*}
		\begin{split}
			x_0 & =u^2+2nu-n^2\\
			y_0 & =u^2-2nu-n^2\\
		\end{split}
		\qquad \qquad
		\begin{split}
			z_0 & =u^2+n^2\\
			w_0 & =2v\\
		\end{split}
	\end{equation*}
	in which $x_0y_0z_0w_0 \ne 0$, is a rational solution for the diagonal quartic surface 
	\begin{equation}\label{equ21}
		x^4+y^4=2z^4+2n^2w^4
	\end{equation}
	
\end{prop}

\begin{proof}
	$n$ is a congruent number thus using Proposition \ref{pro1} part(i), there exist $x_0, y_0, z_0, w_0 \in \mathbb{Z^+}$ satisfying the equations
	\begin{equation}\label{equ22}
		\begin{split}
			z_0^2+nw_0^2 & =x_0^2\\
			z_0^2-nw_0^2 & =y_0^2\\
		\end{split}
	\end{equation}
	we have: 
	\[
	(z_0^2+nw_0^2)^2+(z_0^2-nw_0^2)^2= 2z_0^4+2n^2w_0^4
	\]
	Therefore, $(x_0,y_0,z_0,w_0)$ is a solution for the equation $x^4+y^4=2z^4+2n^2w^4$. Again from equations \ref{equ22} we have:
	\[
	(z_0^2+nw_0^2)^2-(z_0^2-nw_0^2)^2= n(2z_0w_0)^2
	\]
	putting $t_0=2z_0w_0$, as a result $(x_0,y_0,t_0)$ is a solution for the Diophantine equation $x^4-y^4=nt^2$. In which $x_0y_0t_0 \ne 0$.\\
	Since $n$ is a congruent number by Proposition \ref{pro1} part(iii), the equation $v^2=u^3-n^2u$ has a solution $(u,v)$ which $v \ne 0$. Thus by Proposition \ref{pro5} 
	\begin{equation}\label{equ37}
		\begin{split}
			x_0 & =u^2+2nu-n^2\\
			y_0 & =u^2-2nu-n^2\\
			t_0 & =4v(u^2+n^2)\\
		\end{split}
	\end{equation}
	is a solution for the equation $x^4-y^4=nt^2$. Using equations \ref{equ37} we have
	\begin{equation*}
		\begin{split}
			z_0^2 & = \frac{x_0^2+y_0^2}{2} = \frac{(u^2+2nu-n^2)^2+(u^2-2nu-n^2)^2}{2} =(u^2+n^2)^2\\
			w_0^2 & = \frac{x_0^2-y_0^2}{2n} = \frac{8nu(u^2+n^2)}{2n}=4u(u^2+n^2)= 4v^2\\
		\end{split}
	\end{equation*}
	then $z_0=u^2+n^2$ and $w_0=2v$. The rational solution of equation \ref{equ21} is:
	\begin{equation*}
		\begin{split}
			x_0 & =u^2+2nu-n^2\\
			y_0 & =u^2-2nu-n^2\\
		\end{split}
		\qquad \qquad
		\begin{split}
			z_0 & =u^2+n^2 \\
			w_0 & =2v \\
		\end{split}
	\end{equation*}
	
\end{proof}

\begin{lemma}\label{lem1}
	Let $n$ be a congruent number. The following diagonal quartic surface
	\begin{equation}\label{equ38}
		x^4+y^4=z^4+2n^2w^4
	\end{equation}
	has a rational parametric solution in the form:
	\begin{equation*}
		\begin{split}
			x &=X^2\\
			y &=nY^2\\
		\end{split}
		\qquad \qquad
		\begin{split}
			z &=K\\
			w &=XY\\
		\end{split}
	\end{equation*}
	
\end{lemma}

\begin{proof}
	Since $n$ is a congruent number using Proposition \ref{pro1} part $(ii)$, there are positive integers $X$ , $Y$ and $K$ such that
	\[
	X^4-n^2Y^4=K^2
	\]
	by squaring the above equation we have
	\[
	(X^2)^4-2n^2X^4Y^4+(nY^2)^4=K^4
	\]
	or
	\[
	(X^2)^4+(nY^2)^4=K^4+2n^2X^4Y^4
	\]
	so it is evident that
	\begin{equation*}
		\begin{split}
			x &=X^2\\
			y &=nY^2\\
		\end{split}
		\qquad \qquad
		\begin{split}
			z &=K\\
			w &=XY\\
		\end{split}
	\end{equation*}
	is the rational parametric solution for equation \ref{equ38}.\\ 
\end{proof}

\begin{lemma}\label{lem2}
	There is a rational parametric solution for the following diagonal quartic surface. 
	\begin{equation}\label{equ28}
		sx^4+\frac{4}{s}y^4=z^4-w^4
	\end{equation}
	In which $s$ is a nonzero rational parameter.	
	
\end{lemma}

\begin{proof}
	Considering the identity : 
	\begin{equation}\label{equ30}
		(a+b)^4-(a-b)^4=8a^3b+8ab^3
	\end{equation}
	putting
	\begin{equation*}
		a= \frac{1}{s} t^4 \qquad, \qquad b=2
	\end{equation*}
	Then
	\begin{equation*}
		\begin{split}
			8a^3b+8ab^3 & = 8 (\frac{t^4}{s})^3 \times 2 + 8 (\frac{t^4}{s}) \times 2^3\\
			& = s (\frac{2t^3}{s})^4 + \frac{4}{s} (2t)^4\\
		\end{split}
	\end{equation*}
	Therefore 
	\begin{equation*}
		\begin{split}
			x & =\frac{2t^3}{s}\\
			y & = 2t\\
		\end{split}
		\qquad \qquad
		\begin{split}
			z & = a+b = \frac{t^4}{s} + 2 \\
			w & = a-b = \frac{t^4}{s} - 2 \\
		\end{split}
	\end{equation*}\\
	is a rational parametric solution for diagonal quartic surface \ref{equ28}.
\end{proof}

\begin{remark}\label{rem2}
	We are interested in two specific cases of equation \ref{equ28}, as the parameter $s$ takes the values of 1 and 2. 
	For the case $s=1$, the equation is
	\begin{equation}\label{equ29}
		x^4+4y^4=z^4-w^4
	\end{equation}
	and the rational parametric solution for this equation is
	\[
	x=2t^3 \quad, \quad y=2t \quad, \quad z=t^4+2 \quad, \quad w=t^4-2
	\]
	And for the case $s=2$, the equation is
	\begin{equation}\label{equ35}
		2x^4+2y^4=z^4-w^4
	\end{equation}
	and the rational parametric solution for this equation is 
	\[
	x=t^3 \quad, \quad y=2t \quad, \quad z=\frac{t^4}{2}+2 \quad, \quad w=\frac{t^4}{2}-2
	\]
\end{remark}

\section{Proof of the Theorems}

\begin{proof}[Proof of theorem 1.1.]
	Proof for equation \ref{equ16} and equation \ref{equ17} are presented. For the rest of the equations the proof is in the same fashion.\\
	
	\underline{Equation \ref{equ16}}. According to the Proposition \ref{pro2} part $(i)$, $n = a^4+4b^4$ is a congruent number. Hence from Proposition \ref{pro1} part $(i)$, the $(\bar{x},\bar{y},\bar{z},\bar{w})$ satisfies the equation 
	\begin{equation*}
		\bar{x}^2+n\bar{y}^2=\bar{z}^2
	\end{equation*}
	\begin{equation*}
		\bar{x}^2-n\bar{y}^2=\bar{w}^2
	\end{equation*}
	in which by Proposition \ref{pro2} part $(i)$,
	\begin{equation}\label{18}
		\begin{split}
			\bar{x} & =a^8+24a^4 b^4+16 b^4\\
			\bar{y} & = 4ab(4b^4-a^4)\\
			\bar{z} & = a^8+32a^2b^6-8a^4b^4+8a^6b^2+16b^8\\
			\bar{w} & = a^8-32a^2b^6-8a^4b^4-8a^6b^2+16b^8\\
		\end{split}
	\end{equation}
	But from Remark \ref{rem1} 
	\begin{equation}\label{equ26}
		U=(\frac{\bar{x}}{\bar{y}})^2, \qquad V=\frac{\bar{x}\bar{z}\bar{w}}{\bar{y}^3}
	\end{equation}
	satisfies the elliptic curve equation $V^2=U^3-n^2U$.
	Thus according to the Proposition \ref{pro3} the values 
	\begin{equation}
		\begin{split}\label{19}
			x_1 & =U^2+2nU-n^2\\
			y_1 & =U^2-2nU-n^2\\
			z_1 & =U^2+n^2\\
			w_1 & =2V
		\end{split}
	\end{equation}
	is the solution of the equation   
	\begin{equation}\label{equ3}
		x^4+y^4=2z^4+2n^2w^4
	\end{equation}
	substituting $n=a^4+4b^4$ in equation \eqref{equ3}, we have
	\begin{equation*}
		x_1^4+y_1^4=2z_1^4+2(a^4+4b^4)^2w_1^4
	\end{equation*}
	through slight computation
	\begin{equation*}
		x_1^4+y_1^4-(2abw_1)^4=2z_1^4+2(a^2w_1)^4+2(2b^2w_1)^4
	\end{equation*}
	Therefore,
	\begin{equation*}
		\begin{split}
			x &=x_1 \\
			y &=y_1 \\
			z &=2abw_1 \\
		\end{split}
		\qquad \qquad
		\begin{split}
			u &=z_1 \\	
			v &=a^2w_1 \\
			w &=2b^2w_1\\
		\end{split}
	\end{equation*}
	is a solution for the following equation:
	\[
	x^4+y^4-z^4=2u^4+2v^4+2w^4
	\]
	Finally by substituting the values of the equations \ref{equ26} in \ref{19}, the rational parametric solution of the equation \ref{equ16} is obtained.
	\begin{equation*}
		\begin{split}
			x & =x_1= (\frac{\bar{x}}{\bar{y}})^4+2(a^4+4b^4)(\frac{\bar{x}}{\bar{y}})^2-(a^4+4b^4)^2\\
			y & =y_1= (\frac{\bar{x}}{\bar{y}})^4-2(a^4+4b^4)(\frac{\bar{x}}{\bar{y}})^2-(a^4+4b^4)^2\\
			z & = 2abw_1=4ab(\frac{\bar{z}\bar{w}\bar{x}}{\bar{y}^3})\\
			u & =z_1= (\frac{\bar{x}}{\bar{y}})^4+(a^4+4b^4)^2\\
			v & =a^2w_1= 2a^2(\frac{\bar{z}\bar{w}\bar{x}}{\bar{y}^3})\\
			w & =2b^2w_1= 4b^2(\frac{\bar{z}\bar{w}\bar{x}}{\bar{y}^3})
		\end{split}
	\end{equation*}
	It is clear that by using the values of $\bar{x}$, $\bar{y}$, $\bar{z}$ and $\bar{w}$ from equation \eqref{18} the rational parametric solution in terms of $a$ and $b$ is obtained.\\
	
	\underline{Equation \ref{equ17}}. Using Proposition \ref{pro2} part $(i)$, $n = a^4+4b^4$ is a congruent number. Thus based on the Proposition \ref{pro1} part $(i)$, there exist positive integers $\bar{x}, \bar{y}, \bar{z}, \bar{w}$ such that $(\bar{x},\bar{y},\bar{z},\bar{w})$ satisfies the following equations
	\begin{equation*}
		\begin{split}
			\bar{x}^2+n\bar{y}^2=\bar{z}^2\\
			\bar{x}^2-n\bar{y}^2=\bar{w}^2\\
		\end{split}
	\end{equation*}
	in which by Proposition \ref{pro2} part $(i)$,
	\begin{equation*}
		\begin{split}
			\bar{x} & =a^8+24a^4 b^4+16 b^4\\
			\bar{y} & = 4ab(4b^4-a^4)\\
			\bar{z} & = a^8+32a^2b^6-8a^4b^4+8a^6b^2+16b^8\\
			\bar{w} & = a^8-32a^2b^6-8a^4b^4-8a^6b^2+16b^8\\
		\end{split}
	\end{equation*}
	From Remark \ref{rem1}, by putting
	\begin{equation}\label{equ27}
		X= \bar{x}, \quad Y= \bar{y}, \quad K= \bar{z} \bar{w},
	\end{equation}
	the following equation holds:
	\[
	X^4-n^2Y^4=K^2
	\]
	Thus using Lemma \ref{lem1} 
	\begin{equation*}
		\begin{split}
			x & = X^2\\
			y & = nY^2\\
		\end{split}
		\qquad \qquad
		\begin{split}
			z & = K\\
			t & = XY\\
		\end{split}
	\end{equation*}
	is the solution of the following equation
	\[
	x^4+y^4=z^4+2n^2t^4
	\]
	substituting $n=a^4+4b^4$, we have:
	\[
	x^4+y^4=z^4+2(a^4+4b^4)^2t^4
	\]
	by slight computation
	\[
	x^4+y^4-z^4=(2abt)^4+2(2b^2t)^4+2(a^2t)^4
	\]
	Therefore 
	\begin{equation*}
		\begin{split}
			x & = X^2\\
			y & = (a^4+4b^4)Y^2\\
			z & = K\\
		\end{split}
		\qquad \qquad
		\begin{split}
			u & = 2abt=2abXY\\
			v & = 2b^2t=2b^2XY\\
			w & = a^2t=a^2XY\\
		\end{split}
	\end{equation*}
	is the solution of equation \ref{equ17}. Using the equations \ref{equ27}, we have:
	\begin{equation*}
		\begin{split}
			x & = \bar{x}^2\\
			y & = (a^4+4b^4)\bar{y}^2\\
			z & = \bar{w}\bar{z}\\
		\end{split}
		\qquad \qquad
		\begin{split}
			u & = 2ab(\bar{x}\bar{y})\\
			v & = 2b^2(\bar{x}\bar{y})\\
			w & = a^2(\bar{x}\bar{y})\\
		\end{split}
	\end{equation*}
	Which is the solution of the following equation: 
	\[
	x^4+y^4-z^4=u^4+2v^4+2w^4
	\]
	It is clear that through substituting the values of $\bar{x}$, $\bar{y}$, $\bar{z}$ and $\bar{w}$ from equation \eqref{18}, the rational parametric solution in terms of $a$ and $b$ is obtained.
\end{proof}

\begin{proof}[Proof of theorem 1.2.]
	Based on the assumption of the theorem, equation \eqref{equ13} has a rational parametric solution. Through slight computation this rational parametric solution can be written in the form of polynomials with integer coefficients. Since for each parameter $t$ the value $f(a_1(t),a_2(t))$ is a congruent number, hence $\sum_{i=1}^{m} k_i {(b_i(t))}^4$ will also be a congruent number. Applying the same strategy presented in the theorem \ref{th1}, we can exhibit a rational parametric solution for the following equation
	\begin{equation}\label{equ14}
		x^4+y^4=2z^4+2(\sum_{i=1}^{m} k_i {(b_i(t))}^4)^2w^4
	\end{equation}
	Since $(\sum_{i=1}^{m} k_i {(b_i(t))}^4)^2$ has $\binom{m+1}{2}$ terms, therefore the equation \eqref{equ14} contains $\binom{m+1}{2}+3$ terms. This will complete the proof:
	\begin{equation*}
		\binom{m+1}{2}+3 = \frac{1}{2}(m^2+m+6)
	\end{equation*}
	
\end{proof}
\begin{remark}
	Implementing Theorem \ref{th2} for the case $m=2$, we are able to parameterize numerous equations in addition to the equations mentioned in Theorem \ref{th1}. For instance using Lemma \ref{lem2}, lets consider the rational parametric solution of the family of degree 4 surfaces as below:
	\[
	a^4-b^4=sc^4+ \frac{4}{s}\;d^4
	\] 
	which are 
	\begin{equation*}
		\begin{split}
			a & = \frac{1}{s}\;t^4+2 \\
			b & = \frac{1}{s}\;t^4-2 \\
		\end{split}
		\qquad \qquad
		\begin{split}
			c & = \frac{2}{s}\;t^3 \\
			d & = 2t \\
		\end{split}
	\end{equation*}
	applying exactly the approach mentioned in the proof of Theorem \ref{th2}, we can easily obtain a rational parametric solution for the family of following Diophantine equations:
	\begin{equation*}
		x^4+y^4-s^2z^4=2u^4+2v^4+2s^2w^4
	\end{equation*}
	
\end{remark}

\begin{proof}[Proof of theorem 1.3.]
	Considering the algebraic diagonal varieties as below:
	\begin{equation}\label{equ15}
		f(x,y)=\sum_{i=1}^{m} k_i {x_i}^4
	\end{equation}
	in which $f(x,y)$ is chosen from one of the expressions of equations \ref{equ36}. It is easy to select the coefficients $k_is$ such that equation \eqref{equ15} becomes symmetric. According to \cite{2} there is a rational parametric solution for the symmetric diagonal varieties. Using Theorem \ref{th2} the proof is completed.
	
\end{proof}

\begin{proof}[Proof of theorem 1.4.]
	without loss of generality, every rational parametric solution of equation \ref{equ31} can be expressed in the form of polynomials with integer coefficients. Thus, we may consider each of the $F_i(x,y)$ as a congruent number. Using the equation \ref{equ21} from Proposition \ref{pro3}, we have : 
	\begin{equation}\label{equ33}
		2n^2= (\frac{x}{w})^4+ (\frac{y}{w})^4-2(\frac{z}{w})^4
	\end{equation} 
	also from Lemma \ref{lem1} 
	\begin{equation}\label{equ34}
		2n^2= (\frac{x}{w})^4+ (\frac{y}{w})^4-(\frac{z}{w})^4
	\end{equation} 
	Through substituting $F_i(x_i,y_i)$ instead of $n$ in equations \ref{equ33} or \ref{equ34} and by equating them, the rational parametric solution for the Diophantine chain \ref{equ32} is obtained.
	
\end{proof}

\begin{corol}Each of the following Diophantine equations has a rational parametric  solution: 
	\begin{equation*}
		\begin{split}
			x^4+y^4+2z^4 &= u^4+v^4+2w^4\\
			x^4+y^4+z^4 &= u^4+v^4+w^4\\
			x^4+y^4+z^4 &= u^4+v^4+2w^4
		\end{split}
	\end{equation*}

\end{corol}

\begin{proof}
	The proof is the direct result of Lemma \ref{lem2} and Theorem \ref{th4}.  
\end{proof}

\section{Closing Comments}

Through applying Proposition \ref{pro3} and Lemma \ref{lem1} it is concluded that for any positive integer like $m$ each of the following symmetric Diophantine chains has infinitely many rational solutions
\begin{equation}\label{equ23}
	\phi(x_1,y_1,z_1) =\phi(x_2,y_2,z_2)= \dots =\phi(x_m,y_m,z_m)
\end{equation}
\begin{equation}\label{equ24}
	\psi(x_1,y_1,z_1) =\psi(x_2,y_2,z_2)= \dots =\psi(x_m,y_m,z_m)
\end{equation}
in which 
\begin{equation*}
	\begin{split}
		\phi(x,y,z) & = x^4+y^4-2z^4\\
		\psi(x,y,z) & = x^4+y^4-z^4\\
	\end{split}
\end{equation*}

Lets consider the rational parametric solution of the equations \ref{equ23} and \ref{equ24}. Suppose that the following Diophantine chain has a rational parametric solution
\begin{equation}\label{equ25}
	f(x_1,y_1)= f(x_2,y_2)= \dots = f(x_m,y_m)
\end{equation}
in which f(a,b) is equal to one the values:
\[
a^4+4b^4 ,\ \qquad 2a^4+2b^4 ,\ \qquad a^4-b^4,
\]
Consequently, according to Theorem \ref{th4}, there exists a rational parametric solution for the Diophantine chains \ref{equ23} and \ref{equ24}. But the challenge is to find a rational parametric solution for the Diophantine chain \ref{equ25}.

For the case $m=2$ using the Lemma \ref{lem2}, there exists a rational parametric solution for the Diophantine chain \ref{equ25}. But for the case $m=3$ the problem becomes tremendously difficult. For instance if $f(a,b)=2a^4+2b^4$, no rational parametric solution has been found so far. And for the case $f(a,b)=a^4-b^4$ there are only infinitely many rational points, and no rational parametric solution has been found yet \cite{3}.
Though there are infinitely many rational points for the Diophantine chains \ref{equ23} and \ref{equ24} and also there exists a rational parametric solution for the case $m=2$, the following conjecture is stated:\\

\textsc{Conjecture.} For any positive integer like $m$, the Diophantine chains of \ref{equ23} and \ref{equ24} have a rational parametric solution.

\end{document}